\documentclass{article}
\usepackage[utf8]{inputenc}
\usepackage{pgf,tikz,tkz-graph}
\usepackage{fullpage}
\usepackage{hyperref,verbatim}
\usepackage{enumitem,bm}
\usepackage{url}
\usepackage{listings}
\usepackage{amsthm,amsmath,amssymb}
\usepackage{url,pdfpages,xcolor,framed,color}
\usepackage{a4wide}

\usepackage{array}

\newtheorem{defi}{Definition}

\newtheorem{thr}[defi]{Theorem}
\newtheorem{lem}[defi]{Lemma}
\newtheorem{obs}[defi]{Observation}

\newtheorem{claim}[defi]{Claim}
\newcommand*{\myproofname}{Proof}
\newenvironment{claimproof}[1][\myproofname]{\begin{proof}[#1]}{\end{proof}}
\def\C{\mathcal{C}}
\def \d{\, \mathrm{d}}
\def\vc{\overrightarrow}

\title{Resolution of Yan's conjecture on entropy of graphs}

\author{
Stijn Cambie\thanks{This author was unemployed when performing this research.
E-mail: {\tt stijn.cambie@hotmail.com}.}\and
Matteo Mazzamurro\thanks{
Department of Computer Science,
University of Warwick, UK. E-mail: {\tt  matteo.mazzamurro@warwick.ac.uk}. This paper is partly funded by: EPSRC Centre for Doctoral Training in Urban Science and Progress (EP/L016400/1), EPSRC DTP (EP/N509796/1).}
}

\begin{document}

\maketitle

\begin{abstract}
     The first degree-based entropy of a graph is the Shannon entropy of its degree sequence normalized by the degree sum. In this paper, we characterize the connected graphs with given order $n$ and size $m$ that minimize the first degree-based entropy whenever $n-1 \le m \le 2n-3,$ thus extending and proving a conjecture by Yan.
\end{abstract}

\section{Introduction}

The first-degree based graph entropy and the Shannon entropy of other graph invariants have attracted significant attention in organic chemistry, as measures of uniformity of a graph's structural aspect of interest \cite{mowshowitz1968entropy,bonchev1983chemical,dehmer2009entropy}. Although Shannon entropy is conceptually and computationally simple, its careful and context-informed normalisation and interpretation poses some challenges. Determining the range of values the Shannon entropy of a graph invariant can take is a non-trivial task as it may depend on the presence of structural constraints on the graph \cite{dehmer2012extremal}. Solving this issue requires first and foremost the identification of the measure's extremal values for graphs satisfying natural constraints.

In~\cite{CM22+}, we determined the minimum first degree-based entropy among all graphs with a given size. Here the extremal graphs are precisely the colex graphs. In this paper, we do so for connected graphs with given size $m$ and order $n$, which we call $(n,m)$-graphs, for the case when $n-1 \leq m \leq 2m-3$. 
This problem was first presented by Yan \cite{yan2021topological}, who conjectured that the degree sequence of the graph minimising the first-degree based graph entropy when $m\geq n+9$ is $(n-1,m-n+2,2^{m-n+1},1^{2n-m-3}).$ Here we refine this conjecture, first of all by noticing that such degree sequence is only possible when $m\leq 2n-3$, and then showing how, with some adjustments, the conjecture can be extended to the range $n-1\leq m\leq 2n-3$. We finally proceed to prove it. The extremal graphs, i.e. the graphs minimizing the entropy among all $(n,m)$-graphs, are presented in Table~\ref{tbl:overview}.

	\begin{table}[h!]
		\centering
		\begin{tabular}{ | m{3cm} | m{7cm} | }
			\hline
			&\\
			\begin{minipage}[t]{3cm}
				$m=n-1+a$\\
				$0 \le a \le n-2$\\
				$a \not \in \{3,5,6\}$
			\end{minipage}
			&\begin{minipage}[t]{7cm}
				\centering
				\begin{tikzpicture}{
					
					\foreach \x in {0,-0.4,1,1.6,2.4,2.8}{\draw[thick] (\x,2) -- (0.6,0);}
					\draw[thick] (2,0) -- (0.6,0);
					\foreach \x in {1.6,2.4,2.8}{\draw[thick] (\x,2) -- (2,0);}
					\draw[dotted] (-0.4,2) -- (1,2);
					\draw[dotted] (1.6,2) -- (2.8,2);
					\foreach \x in {0,-0.4,1,1.6,2.4,2.8}{\draw[fill] (\x,2) circle (0.1);}
					\foreach \x in {0.6,2}{\draw[fill] (\x,0) circle (0.1);}
					
					\coordinate [label=center:\large \textbf{$a$}] (A) at (2.2,2.3);
					\coordinate [label=center:\large \textbf{$n-2-a$}] (A) at (0.3,2.3);
					
				}
				\end{tikzpicture}	
			\end{minipage}
			\\
			&\\
			\hline
			
			&\\
			\begin{minipage}[t]{3cm}
				$m=n+2$\\
			\end{minipage}
			&\begin{minipage}[t]{7cm}
				\centering
				\begin{tikzpicture}{

					\foreach \x in {0.1,-0.3,1.1}{\draw[thick] (\x,1) -- (0.6,0);}
					\draw[thick] (0.6,0) -- (1.6,0)--(1.6,-1)--(0.6,-1) -- cycle;
					\draw[thick] (0.6,0) --(1.6,-1);
					\draw[thick] (1.6,0)--(0.6,-1);
					\draw[dotted] (-0.3,1) -- (1,1);
					
					\foreach \x in {0.1,-0.3,1.1}{\draw[fill] (\x,1) circle (0.1);}
					\foreach \x in {0.6,1.6}{\draw[fill] (\x,0) circle (0.1);}
					\foreach \x in {0.6,1.6}{\draw[fill] (\x,-1) circle (0.1);}
					
					\coordinate [label=center:\large \textbf{$n-4$}] (A) at (0.4,1.3);
					
				}
				\end{tikzpicture}	
			\end{minipage}
			\\
			&\\
			\hline
			
			&\\
			\begin{minipage}[t]{3cm}
				$m=n+4$\\
			\end{minipage}
			&\begin{minipage}[t]{7cm}
				\centering
			
					\begin{tikzpicture}{
						\foreach \x in {0.1,-0.3,1.1}{\draw[thick] (\x,1) -- (0.6,0);}
						\draw[thick] (1.4,-1.4) --(0.6,-1)--(0.6,0) -- (1.6,0)--(2,-0.8);
						\draw[thick] (0.6,0) --(2,-0.8) --(0.6,-1)-- (1.6,0)--(1.4,-1.4) -- cycle;
						\draw[dotted] (-0.3,1) -- (1,1);
						
						\foreach \x in {0.1,-0.3,1.1}{\draw[fill] (\x,1) circle (0.1);}
						\foreach \x in {0.6,1.6}{\draw[fill] (\x,0) circle (0.1);}
						\draw[fill] (0.6,-1) circle (0.1);
						\draw[fill] (1.4,-1.4) circle (0.1);
						\draw[fill] (2,-0.8) circle (0.1);
						
						\coordinate [label=center:\large \textbf{$n-5$}] (A) at (0.4,1.3);
						
					}
					\end{tikzpicture}	
				 \hspace{0.5 cm}	
				\begin{tikzpicture}{
					
					\foreach \x in {0,-0.4,1,1.6,2.4,2.8}{\draw[thick] (\x,2) -- (0.6,0);}
					\draw[thick] (2,0) -- (0.6,0);
					\foreach \x in {1.6,2.4,2.8}{\draw[thick] (\x,2) -- (2,0);}
					\draw[dotted] (-0.4,2) -- (1,2);
					\draw[dotted] (1.6,2) -- (2.8,2);
					\foreach \x in {0,-0.4,1,1.6,2.4,2.8}{\draw[fill] (\x,2) circle (0.1);}
					\foreach \x in {0.6,2}{\draw[fill] (\x,0) circle (0.1);}
					
					\coordinate [label=center:\large \textbf{$5$}] (A) at (2.2,2.3);
					\coordinate [label=center:\large \textbf{$n-7$}] (A) at (0.3,2.3);
					
				}
				\end{tikzpicture}	
			\end{minipage}
			\\
			&\\
			\hline
			
			&\\
			\begin{minipage}[t]{3cm}
				$m=n+5$\\
			\end{minipage}
			&\begin{minipage}[t]{7cm}
				\centering
				\begin{tikzpicture}{
					\foreach \x in {0.1,-0.3,1.1}{\draw[thick] (\x,1) -- (0.6,0);}
					\draw[thick] (0.6,0) -- (1.6,0)--(2,-0.8)--(1.4,-1.4) --(0.6,-1)-- cycle;
					\draw[thick] (0.6,0) --(2,-0.8) --(0.6,-1)-- (1.6,0)--(1.4,-1.4) -- cycle;
					\draw[dotted] (-0.3,1) -- (1,1);
					
					\foreach \x in {0.1,-0.3,1.1}{\draw[fill] (\x,1) circle (0.1);}
					\foreach \x in {0.6,1.6}{\draw[fill] (\x,0) circle (0.1);}
					\draw[fill] (0.6,-1) circle (0.1);
					\draw[fill] (1.4,-1.4) circle (0.1);
					\draw[fill] (2,-0.8) circle (0.1);
					
					\coordinate [label=center:\large \textbf{$n-5$}] (A) at (0.4,1.3);
					
				}
				\end{tikzpicture}	
			\end{minipage}
			\\
			&\\
			\hline
		\end{tabular}
		\caption{Overview of extremal $(n,m)$-graphs minimizing the entropy}\label{tbl:overview}
	\end{table}

Let us now start by formally defining the measure of interest. Here the logarithm will always denote the natural logarithm.

\begin{defi}
    The first degree-based entropy of a graph $G$ with degree sequence $(d_i)_{1 \le i \le n}$ and size $m$ equals
    $$I(G)=-\sum_{i=1}^n \frac{d_i}{2m}\log\left( \frac{d_i}{2m} \right).$$
\end{defi}

If we let $f(x)=x \log(x)$ 
    and $h(G)=\sum_i f(d_i)=\sum_i d_i \log(d_i)$,
    then we have $I(G)=\log(2m)-\frac{1}{2m}h(G).$
    Thus, determining the minimum of $I(G)$ is equivalent to determining the maximum of $h(G).$

By~\cite[Theorem~4]{yan2021topological}, we know that the graph maximizing $h(G)$ among all $(n,m)$-graphs is a threshold graph. This implies in particular that the graph has a universal vertex $v$ with degree $n-1$.
Now $G \backslash v$ is a $(m-n+1,n-1)$-graph.
Taking into account that $d_G(u)=d_{G \backslash v}(u)+1$ for every vertex $u \in V \backslash v,$ we note that it is sufficient to find the $(m-n+1,n-1)$-graph maximizing $h_1(G)$, where $h_1(G)$ is formed by taking into account that the original degrees are larger by one.
We extend this idea towards the setting where there are $c$ universal vertices initially.
Then, we compute the extremal graphs maximizing the related function $h_c(G)$ given only the size (and fixed large order essentially, as explained in Subsection~\ref{subsec:prelimanaryresults}).
We do so by induction. In Section~\ref{sec:extr_graphs_small_size} we compute the extremal graphs for small size. These are the base cases for the induction.
Then by taking a vertex of minimum degree and relating $h_c(G)$ with $h_c(G \backslash v)$,
we perform the induction in Section~\ref{sec:extr_graphs_hc_given_m}. 
Besides a few exceptions, the extremal graphs turn out to be the star, contrary to the extremal graphs for $h(G)$ when only the graph size is given, for which the extremal graphs are colex graphs, see~\cite{CM22+}. 
The precise statement is formulated in Theorem~\ref{thr:main_extr_h_c}.
At the end of the section, in Subsection~\ref{subsec:min_entropy_given(n,m)}, we apply Theorem~\ref{thr:main_extr_h_c} to characterize the graphs minimizing the entropy among $(n,m)$-graphs when $n-1 \le m \le 2n-3$, thus proving an extended version of the conjecture formulated by Yan \cite[Conj.~6]{yan2021topological}.

The main ideas of the proof are given in Section~\ref{sec:extr_graphs_hc_given_m}. Some necessary tools and computations are gathered in Subsection~\ref{subsec:prelimanaryresults} and Section~\ref{sec:comp_lemmas}.

\subsection{Definitions}\label{subsec:prelimanaryresults}

In this paper, we will express the entropy in terms of other functions and use help functions in the computations. These are defined here.

\begin{defi}
    For any constant $c \ge0,$ we define the function $f_c(x)=(x+c)\cdot \log (x+c).$
    For a graph $G$ with degree sequence $(d_i)_{1 \le i \le n}$, we define $h_c(G)=\sum_i f_c(d_i).$
    When $c=0$, we just write $h(G)$ for $h_0(G)=\sum_i d_i \log(d_i).$
\end{defi}

When $c\ge 2$, the function $h_c(G)$ depends on the number of vertices as well, since isolated vertices contribute $f_c(0)=c \log c >0.$ 
Thus, we will compare graphs with a different order by extending the order, i.e. add isolated vertices in such a way that the graphs have the same order.
We could have defined $h_c^N(G)=\sum_i f_c(d_i) + (N-n)f_c(0)$ to do so, but preferred to keep the notation light. 

We remark here that it will be sufficient to focus on connected graphs.

\begin{obs}
	When omitting the isolated vertices, the graph maximizing $h_c(G)$ among all graphs of size $m$ is a connected graph. For this, note that identifying two vertices in different components with strictly positive degrees $d_u, d_v$ leads to an increase of the value $h_c(G)$ since $f_c$ is a strictly convex function, i.e. $f_c(d_u+d_v)+f_c(0)> f_c(d_u)+f_c(d_v).$
\end{obs}

In some proofs, we will also make use of the following function.

\begin{defi}\label{defi:Delta_c}
   The function $\Delta_c$ is defined by $\Delta_c(x)=f_c(x)-f_c(x-1)=1+\int_{x+c-1}^{x+c} \log t \, \mathrm{d}t$.
\end{defi}
Note that $\Delta_c$ is a strictly concave, increasing function.

\section{Extremal graphs for small size}\label{sec:extr_graphs_small_size}

In this section, we compute the extremal graphs maximizing $h_1(G)$ for $m \le 10$ and for $h_c(G)$ with $c \ge 2$ for $m\le 6.$ 

\begin{lem}\label{lem:c=1_basecases}
    For $m \le 10$, among all graphs with $m$ edges, we have that $h_1(G)$ is maximized by
    $$G = \begin{cases}
K_{1,m}=S_{m+1} & \mbox{ if } m \not\in \{3,4,6\} \\
K_3 & \mbox{ if } m=3\\
K_4^-=\C(5,3) \mbox{ and }S_6& \mbox{ if } m=5\\
K_4 & \mbox{ if } m=6.\\
\end{cases}$$
\end{lem}

\begin{proof}
    A computer program can verify this claim. Since $h_1(G)$ only depends on the degree sequence of the graph, for a given $m\leq 10$, it is enough to list all degree sequences of graphs of size $m$ and then compute $h_1$ for each sequence. To list all degree sequences, it is sufficient to list all integer partitions of $2m$ and then establish which of these are valid degree sequences using one of several existing criteria (see, e.g., \cite{sierksma1991seven}). For example, one can use the function \verb|parts()| from the R-package \verb|partitions| \cite{hankin2021partitions} to list all partitions of $2m$ and check which ones are degree sequences using \verb|is_graphical()| from the R-package \verb|igraph| \cite{csardi2006igraph} (see Appendix \ref{sec:appendix_code}).
\end{proof}

\begin{lem}\label{lem:cge2_basecases}
    For $c\ge 2$ and $m \le 6$, among all graphs with $m$ edges, we have that $h_c(G)$ is maximized by
    $$G = \begin{cases}
K_{1,m} & \mbox{ if } m \not=3 \\
K_3 & \mbox{ if } m=3.\\
\end{cases}$$
Here one has to take into account isolated vertices when comparing graphs with different order.
\end{lem}
\begin{proof}
    For $m \in \{1,2\}$ nothing needs to be done, as there is only one connected graph of size $m.$
    When $m=3,$ there are precisely $3$ connected graphs and we observe that 
    $$h_c(P_4)=2f_c(1)+2f_c(2)<h_c(S_4)=3f_c(1)+f_c(3)<h_c(K_3)=3f_c(2)+f_c(0).$$
    The first inequality is true due to the strict convexity of the function $f_c.$
    The second inequality is true since $\Delta_c$ (Definition~\ref{defi:Delta_c}) is strictly concave and thus $\Delta_c(3)+\Delta_c(1)<2\Delta_c(2).$
    
    By the inequality of Karamata, it is sufficient to consider the degree sequences of graphs with size $m$ that are not majorized by the degree sequences of other such graphs.
    With a simple computer program\footnote{\url{https://github.com/StijnCambie/EntropyGraphs/blob/main/ExtrG_h_c_forsmallm.py}}, we verify those.
    
    For $m=4$ and $m=5$, these non-majorized degree sequences are respectively $\vc{v^4_1}=\{4, 1, 1, 1, 1\}, \vc{v^4_2}=\{3, 2, 2, 1, 0\}$ and $\vc{v^5_1}=\{5, 1, 1, 1, 1, 1\}, \vc{v^5_2}=\{4, 2, 2, 1, 1, 0\}, \vc{v^5_3}=\{3, 3, 2, 2, 0, 0\}.$
    For $m=6$, these degree sequences are $\vc{v^6_1}=\{6, 1, 1, 1, 1, 1, 1\}, \vc{v^6_2}=\{5, 2, 2, 1, 1, 1, 0\}, \vc{v^6_3}=\{4, 3, 2, 2, 1, 0, 0\}, \vc{v^6_4}= \{3, 3, 3, 3, 0, 0, 0\}.$
    
    Now we verify that $h_c\left(\vc{d}\right)=\sum_i f_c(d_i)$ is always maximized by the first degree sequence.
    
     For $4 \le m\le 6,$ we have $$h_c\left(\vc{v^m_1}\right)-h_c\left(\vc{v^m_2}\right)= \Delta_c(m)+\Delta_c(1)-2\Delta_c(2)\ge \int_{c-1}^{c} \log\left( (t+1)(t+4) \right) - \log\left( (t+2)^2 \right) \d t > 0.$$ The last inequality is true since $(t+1)(t+4)>(t+2)^2$ whenever $t \ge 1$.
       
      For $m \in \{5,6\}$ we analogously have
        \begin{align*}
            h_c\left(\vc{v^m_1}\right)-h_c\left(\vc{v^m_3}\right)&=
            \Delta_c(m)+\Delta_c(m-1)+2\Delta_c(1)-\Delta_c(3)-3\Delta_c(2)\\
            &\ge \int_{c-1}^{c} \log\left( (t+5)(t+4)(t+1)^2 \right) - \log\left( (t+3)(t+2)^3 \right) \d t \\ &> 0.
        \end{align*} 
        
        The last inequality being true since $(t+5)(t+4)(t+1)^2>(t+3)(t+2)^3$ whenever $t \ge 1$. 
    
    For the final case, we have
        \begin{align*}
            h_c\left(\vc{v^6_1}\right)-h_c\left(\vc{v^6_4}\right)&=\Delta_c(6)+\Delta_c(5)+\Delta_c(4)+3\Delta_c(1)-3\Delta_c(3)-3\Delta_c(2)\\
            &=\int_{c-1}^{c} \log\left( (t+6)(t+5)(t+4)(t+1)^3 \right) - \log\left( (t+3)^3(t+2)^3 \right) \d t \\&> 0.
        \end{align*}
        When $c=2$, this can be computed\footnote{It is approximately $0.0629$}.
        For $c \ge 3,$ this is due to $ (t+6)(t+5)(t+4)(t+1)^3 > (t+3)^3(t+2)^3$ for $t \ge 2.$
        Finally, it is also clear that the extremal degree sequences do correspond with the star $S_{m+1}=K_{1,m}.$
\end{proof}

\section{Graphs maximizing $h_c(G)$ given the size}\label{sec:extr_graphs_hc_given_m}

In this section, we prove the following theorem that gives the precise characterization of extremal graphs for $h_c(G)$ where $c \ge 1$ is an integer (for $c=0$, this was done in~\cite{CM22+}). 

\begin{thr}\label{thr:main_extr_h_c}
    Among all graphs with $m$ edges, we have that $h_1(G)$ is maximized by
    $$G = \begin{cases}
K_{1,m}=S_{m+1} & \mbox{ if } m \not\in \{3,4,6\} \\
K_3 & \mbox{ if } m=3\\
K_4^-=\C(5,3) \mbox{ and }S_6& \mbox{ if } m=5\\
K_4 & \mbox{ if } m=6.\\
\end{cases}$$
For any $c\ge 2$, among all graphs with $m$ edges and $n>m$ vertices, we have that $h_c(G)$ is maximized by
    $$G = \begin{cases}
K_{1,m} & \mbox{ if } m \not=3 \\
K_3 & \mbox{ if } m=3.\\
\end{cases}$$
\end{thr}
\begin{proof}
Assume we know the extremal graphs with size at most $m-1.$
By Lemmas~\ref{lem:c=1_basecases} and~\ref{lem:cge2_basecases}, this has been done for $m\le 6$ and for $m\le 10$ when $c=1.$
So we assume $m \ge 7$, and even $m \ge 11$ if $c =1.$
Let $G$ be an extremal graph with size $m$ for which the minimum (non-zero) degree is equal to $b.$
The latter implies that there are at least $b+1$ vertices with degree at least $b$ and thus $m \ge \binom{b+1}{2}.$
Let $v$ be a vertex with degree $b$ and let $d_1,d_2,\ldots, d_b$ be the degrees of the neighbours of $v$.

If $b=1$, we have
\begin{align*}
h_c(G)&=h_c(G \backslash v)+ f_c(1)-f_c(0) + \Delta_c(d_1)\\
&\le h_c(K_{1,m-1})+ f_c(1)-f_c(0) + \Delta_c(m)\\
&=h_c(K_{1,m})
\end{align*}
and equality occurs if and only if $G=K_{1,m}.$

Now assume $b \ge 2.$
Note that $\sum_{i=1}^b d_i \le m+\binom b2$ by the analog of the handshaking lemma since every edge which is not part of $G[N(v)]$ can be counted at most once.
Since $\Delta_c$ is strictly concave, we have 
\begin{align*}
    h_c(G)-h_c(G \backslash v) &= f_c(b)-f_c(0)+\sum_{i=1}^b \Delta_c(d_i)\\
    &\le f_c(b)-f_c(0)+b \cdot \Delta_c\left( \frac{m+\binom{b}{2}}{b}\right)\\
    &:=LHS(m,b,c).
\end{align*}
On the other hand, we also have
\begin{align*}
    h_c(K_{1,m})-h_c(K_{1,m-b}) &= f_c(m)-f_c(m-b)+b \Delta_c\left(1\right)\\
    &:=RHS(m,b,c).
\end{align*}
By computations performed in Section~\ref{sec:comp_lemmas}, we know that the first is smaller than the second, i.e. $LHS(m,b,c)<RHS(m,b,c).$
Now $G \backslash v$ has $m-b$ edges, here $m-b\ge 4$ (for $c \ge 2$) and $m-b \ge 7$ (for $c =1$).
Due to Lemmas~\ref{lem:c=1_basecases} and~\ref{lem:cge2_basecases} we have $h_c(G \backslash v) \le h_c(K_{1,m-b}).$

So we conclude that 
\begin{align*}
    h_c(G)&=h_c(G \backslash v) + f_c(b)-f_c(0)+\sum_{i=1}^b \Delta_c(d_i)\\
    &\le h_c(G \backslash v)+f_c(b)-f_c(0)+b \cdot \Delta_c\left( \frac{m+\binom{b}{2}}{b}\right)\\
    &< h_c(K_{1,m-b})+f_c(m)-f_c(m-b)+b \Delta_c\left(1\right)\\
    &=h_c(K_{1,m}).
\end{align*}

By complete induction, we have the whole characterization.
\end{proof}

\subsection{Proof of Yan's Conjecture}\label{subsec:min_entropy_given(n,m)}

We now prove an extended version of Yan's conjecture~\cite[Conj.6]{yan2021topological}.

\begin{thr}
    When $n \le m \le 2n-3$, the extremal $(n,m)$-graph minimizing the entropy is such that, deleting its universal vertex, one obtains the $(n-1,m-n+1)$-graph $G$ described in Theorem \ref{thr:main_extr_h_c} for the $c=1$ case, i.e. the graph with $m'=m-n+1$ edges maximizing $h_1(G).$
    That is, the extremal $(n,m)$-graph has degree sequence
    $$\begin{cases}
(n-1,m-n+2,2^{m-n+1},1^{2n-m-3}) & \mbox{ if } 2n-3 \ge m \ge n+6 \mbox{ or } m \in \{n,n+1,n+3\}, \\
(n-1,4^{4},1^{n-5}) & \mbox{ if }  m = n+5, \\
(n-1,4^{2},3^2,1^{n-5}) \mbox{ or } (n-1,6,2^5,1^{n-7}) & \mbox{ if }  m = n+4, \\
(n-1,3^{3},1^{n-4}) & \mbox{ if }  m = n+2, \\
(n-1,1^{n-1}) & \mbox{ if }  m = n-1. \\
\end{cases}$$
These graphs are presented in Table~\ref{tbl:overview}.
\end{thr}

\begin{proof}
    By~\cite[Theorem~4]{yan2021topological}, we know that the extremal $(n,m)$-graph is a threshold graph. This implies in particular that it has a universal vertex $v$ with degree $n-1$.
    Now $G'=G \backslash v$ is a $(m-n+1,n-1)$-graph.
    Taking into account that $d_{G}(u)=d_{G'}(u)+1$ for every vertex $u \in V \backslash v,$ we note that 
    $$h(G)=f(n-1)+h_1(G').$$
    Now since $m-n+1 \le n-2,$ we note that the extremal structure for $G'$ is determined in Theorem~\ref{thr:main_extr_h_c} and the conclusion is immediate.
\end{proof}

\section{Computational claims}\label{sec:comp_lemmas}

In this section, we prove that 
$$LHS(m,b,c)=f_c(b) - f_c(0) + b\cdot\left(   f_c\left( \frac{m+\binom{b}{2}}{b}\right) -  f_c\left( \frac{m+\binom{b}{2}}{b}-1\right)\right)$$
    and 
    $$RHS(m,b,c)=f_c(m) - f_c(m - b) + b\cdot(f_c(1) - f_c(0))$$
    satisfy $LHS(m,b,c)<RHS(m,b,c)$ for every $b \ge 2$ and $m \ge \binom{b+1}{2}$ whenever $m \ge 7$ and $c\ge 1$, or $m \ge 4$ and $c\ge 2.$

We do this by means of the following claims.
In Claim~\ref{clm:RHS-LHSincreasing}, we prove that for fixed $b$ and $c$, it is sufficient to prove it for the smallest $m$ in the range.
After that, it is proven in the cases for which $m=\binom{b+1}{2}$ in Claim~\ref{clm:LLvsRL} and for the remaining cases in Claim~\ref{LL>RL_smallm}.

The proofs are mainly computational and there are
alternative computations that lead to the same conclusion.

\begin{claim}\label{clm:RHS-LHSincreasing}
    Fix $b \ge 2$ and $c \ge 1.$
    Then $RHS(m,b,c)-LHS(m,b,c)$ is an increasing function in $m.$
\end{claim}
\begin{claimproof}
    We want to prove that the derivative of this quantity with respect to $m$ is positive. To compute the derivative, taking into account the chain rule and
    $\frac{d}{dx}f_c(x)=\log(x+c)+1$, we have that 
    $$\frac{d}{dm}\left( RHS(m,b,c)-LHS(m,b,c) \right)=
    \log\left(\frac{m + c}{m - b + c}\right)-\log\left(\frac{m +\binom{b}{2}+bc}{m +\binom{b}{2}+bc-b}\right)>0.$$
    The inequality now follows the fact
    whenever $0<b<y<z$, we have $\frac{y}{y-b}>\frac{z}{z-b}$. Here it is enough to take $y=m+c$ and $z=m+\binom{b}{2}+bc$.
\end{claimproof}

\begin{claim}\label{clm:LLvsRL}
    Fix $b \ge 2$ and $c \ge 1.$
    Let $$LL(b,c)= (b+1)f_c(b)-f_c(0)-bf_c(b-1)$$ and $$RL(b,c)= f_c \left( \binom{b+1}{2} \right) -f_c\left( \binom{b}{2} \right)  + b\cdot\left(   f_c(1) -  f_c(0)\right).$$ 
    Then $$LL(b,c)<RL(b,c)$$ if $c=1$ and $b \ge 4$, or $c \ge 2$ and $b\ge 3.$
\end{claim}
\begin{claimproof}
    The cases $1\leq c\leq3$ can be verified directly using the formulae: solving numerically the resulting inequalities in the variable $b$, one finds that the inequality holds as long as $b>3.24$, $b>2.53$, and $b>2.34$ for $c=1$, $c=2$, and $c=3$, respectively. A proof that $RL(b,c)-LL(b,c)$ is increasing in $b$, has been put in Appendix~\ref{sec:appendix_comp}, where an alternative strategy for the verification of this claim has been given.
    
     For $c\geq4$ and $b\ge 3$, write
    \begin{align*}
        RL(b,c)&=f_c\left(\binom{b}{2}+b\right)-f_c\left(\binom{b}{2}\right)+b\Delta_c(1)\\
        &=\sum_{i=1}^b\left[f_c\left(\binom{b}{2} +i\right)-f_c\left(\binom{b}{2}+i-1\right)\right]+b\Delta_c(1)\\
        &=\sum_{i=1}^b\Delta_c\left(\binom{b}{2}+i\right)+b\Delta_c(1),
    \end{align*}
    and
    \begin{align*}
        LL(b,c)&=b\left(f_c(b)-f_c(b-1)\right)+f_c(b)-f_c(0)\\
        &=b\Delta_c(b)+f_c(b)-f_c(0)\\
        &=b\Delta_c(b)+\sum_{i=1}^b\left[f_c(i)-f_c(i-1)\right]\\
        &=b\Delta_c(b)+\sum_{i=1}^b\Delta_c(i).
    \end{align*}    
    Then
    \begin{align}
        RL(b,c)-LL(b,c)&=\sum_{i=1}^b\left[\Delta_c\left(\binom{b}{2}+i\right)-\Delta_c(i)\right]-b\left(\Delta_c(b)-\Delta_c(1)\right)\notag\\
        &> b\left[\Delta_c\left(\binom{b}{2}+b\right)-\Delta_c(b)\right]\label{ineq:RL-LL}   -b\left(\Delta_c(b)-\Delta_c(1)\right)\\
        &=b\left[\Delta_c\left(\binom{b+1}{2}\right)+\Delta_c(1)-2\Delta_c(b)\right]\notag  
    \end{align}
    where inequality \eqref{ineq:RL-LL} follows from $b \ge 2$ and the strict concavity of $\Delta_c(x)$. 
    Then, by definition \ref{defi:Delta_c},
    \begin{equation*}
        RL(b,c)-LL(b,c)\geq b \int_{c-1}^{c}\left( \log\left(t+\binom{b+1}{2}\right)+\log(t+1)-2\log(t+b)\right) \d t\\
    \end{equation*}
    For the integral to be positive, it is enough that, for $c-1<t<c$,
    \begin{equation*}
        \left(t+\binom{b+1}{2}\right)(t+1)-(t+b)^2>0,
    \end{equation*}
    which is equivalent to 
    \begin{equation}
        t(b-2)(b-1)>b(b-1).\label{ineq:tb}
    \end{equation}
    Now, since $b\ge 3$, inequality \eqref{ineq:tb} holds if and only if $t>\frac{b}{b-2}$. Furthermore, $b\ge 3$ also implies $\frac{b}{b-2}\leq3$. But $c\geq4$ so $t>c-1=3\geq\frac{b}{b-2}$. Therefore $RL(b,c)>LL(b,c)$ for $c\geq 4$ and $b\ge 3$ as well. 
\end{claimproof}
     
\begin{claim}\label{LL>RL_smallm}
    It is true that $LHS(7, 3, 1) < RHS(7, 3, 1)$ and $LHS(7, 2, 1) < RHS(7, 2, 1).$\\
    For any $c\ge 2$, it is true that $LHS(4, 2, c) < RHS(4, 2, c).$
\end{claim}

\begin{claimproof}
    The values $RHS(7, 3, 1)-LHS(7, 3, 1)$ and $ RHS(7, 2, 1)-LHS(7, 2, 1)$ are approximately $0.18$ and $0.36$ respectively and thus $LHS(7, 3, 1) < RHS(7, 3, 1)$ and $LHS(7, 2, 1) < RHS(7, 2, 1).$
    The second inequality is equivalent to
    $$\Delta_c(4)+\Delta_c(3)+\Delta_c(1) > 2\Delta_c(2.5)+\Delta_c(2).$$

    This is true for every $c \ge 2$ since
    \begin{align*}
            \int_{c-1}^{c} \log\left( (t+4)(t+3)(t+1) \right) \d t >\int_{c-1}^{c}  \log\left( \left(t+\frac 52 \right)^2(t+2) \right) \d t,
        \end{align*}
        as $(t+4)(t+3)(t+1)>\left(t+\frac 52 \right)^2(t+2)$ for every $t \ge 1.$
\end{claimproof}

\bibliographystyle{abbrv}
\bibliography{bib_entropy}

\appendix
\section{Codes}\label{sec:appendix_code}
The following R code can be used to prove Lemma \ref{lem:c=1_basecases}.

\begin{lstlisting}
#load necessary packages
library(partitions)
library(igraph)

#define function h_c
h_c <- function(x,c){
  #avoid issues with log(0)
  if(c==0){
    x[x==0] <- 1
  }
  ans <- sum((x+c)*log(x+c))
  return(ans)
}

#allow tolerance in the comparison to account for machine precision
tol<-1e-6

#initialize empty list of degree sequences
h_c_max <- list()

for(m in 3:10){
  #find partitions of 2m 
  parts_m <- as.matrix(parts(2*m))
  #select partitions that are valid degree sequences
  is_graphical_m <- apply(parts_m,MARGIN=2,is_graphical)
  graphical_parts_m <- as.matrix(parts_m[,is_graphical_m])
  #find which degree sequence(s) maximize(s) h_c
  h_c_m <- apply(graphical_parts_m,MARGIN=2,h_c,c=1)
  h_c_m_max <- graphical_parts_m[,which((max(h_c_m)-h_c_m)<tol)]
  #add the degree sequence(s) maximizing h_c_m to the list
  h_c_max <- append(h_c_max,list(h_c_m_max))
}
\end{lstlisting}\label{code:h1_small_m}

\section{Precise verification of Claim~\ref{clm:LLvsRL}}\label{sec:appendix_comp}

For easy reference, we restate Claim~\ref{clm:LLvsRL} here.

\begin{claim}\label{clm:LLvsRL_v2}
    Fix $b \ge 2$ and $c \ge 1.$
    Let $$LL(b,c)= (b+1)f_c(b)-f_c(0)-bf_c(b-1)$$ and $$RL(b,c)= f_c \left( \binom{b+1}{2} \right) -f_c\left( \binom{b}{2} \right)  + b\cdot\left(   f_c(1) -  f_c(0)\right).$$ 
    Then $$LL(b,c)<RL(b,c)$$ if $c=1$ and $b \ge 4$, or $c \ge 2$ and $b\ge 3.$
\end{claim}
\begin{claimproof}
     We first prove that the derivative of this quantity with respect to $b$ is positive.
     We have
     \begin{align*}
         \frac{d}{db} RL(b,c) =&\frac{2b+1}{2}\left( \log\left(\binom{b+1}{2} +c \right)+1\right)\\&-\frac{2b-1}{2}\left( \log\left(\binom{b}{2} +c \right)+1\right)+f_c(1)-f_c(0)\\
         \frac{d}{db} LL(b,c) =&(b+1)\left( \log\left(b+c \right)+1\right)\\&-b\left( \log\left(b+c-1 \right)+1\right)+f_c(b)-f_c(b-1)
     \end{align*}
     Combining these two expression, we can write
     \begin{equation}
         \frac{d}{db}\left( RL(b,c)-LL(b,c) \right)=J_1(b,c)+J_2(b,c)+J_3(b,c),
     \end{equation}\label{eq:RL-LL}
     where 
     \begin{align*}
         J_1(b,c)=&\frac{2b-1}{2}\left( \log\left(\binom{b+1}{2} +c \right) -\log\left(\binom{b}{2} +c \right)  -2\log\left(b+c \right)+2\log\left(b+c-1 \right) \right),\\
         J_2(b,c)=&c \left( \log(c+1)-\log(c)-\log(b+c)+\log(b+c-1) \right),\\
         J_3(b,c)=&\log\left(\binom{b+1}{2} +c \right) +\log(c+1) -2\log(b+c).
     \end{align*}
     It is sufficient to prove that $J_i(b,c)\geq0$ for $i\in\{1,2,3\}$ for the above conditions on $b$ and $c$.
      
      By expanding the binomial coefficient, we rewrite $J_1(b,c)$ as 
      \begin{equation*}
          J_1(b,c)=\frac{2b-1}{2}\log\left(\frac{(b^2+b+2c)(b+c-1)^2}{(b^2-b+2c)(b+c)^2}\right).
      \end{equation*}
     Then $J_1(b,c)>0$ if and only if $g(b,c)>0$, where 
     \begin{equation*}
         g(b,c)=(b^2+b+2c)(b+c-1)^2-(b^2-b+2c)(b+c)^2.
     \end{equation*}
     This can be simplified as
     \begin{equation*}
          g(b,c)=2(b-3)bc+2(b-2)c^2+b+2c-b^2.
     \end{equation*}
     Hence we can see that, whenever $b\geq 3$,
    \begin{equation*}
          \frac{\partial}{\partial c}g(b,c)=4(b-2)c+2(b-3)b+2>0,
     \end{equation*}
    and so $g(b,c)$ is a strictly increasing function in $c$ whenever $b\geq 3$. 
    
    For $c=1$ we have that ${g(b,1)=b^2-3b-2>0}$ for $b \ge 4$.
    Similarly, for $c=2$, we get ${g(b,2)=3b^2-3b-12>0}$ for $b \ge 3$. Thus $g(b,c)>0$ (and hence $J_1(b,c)>0$) for $c=1$ and $b\geq 4$, and for $c\geq 2$ and $b\geq 3$.
     
     Analogous arguments apply to $J_2$ and $J_3$.
     
     $J_2(b,c)>0$ since $\log\left(1+\frac 1{c} \right) > \log\left(1+\frac 1{b+c-1} \right)$ when $b-1>0$.
     
     For $J_3(b,c)$ note that 
     \begin{equation*}
         2\left(\binom{b+1}{2} +c \right)(c+1)-2(b+c)^2= (bc-2c-b)(b-1),
     \end{equation*}
     which is non-negative whenever $b,c\ge 3$. 
     
     Combining these results and substituting in equation \eqref{eq:RL-LL}, we have ${\frac{d}{db}\left( RL(b,c)-LL(b,c) \right)>0}$ for $c\geq 3$. 
     
    For the remaining cases, consider the sum $J_2(b,c)+J_3(b,c)$. When $c=1,$
     \begin{equation*}
         J_2(b,1)+J_3(b,1)=\log\left(\frac{2b(b^2+b+2)}{(b+1)^3}\right).
     \end{equation*}
     Now, 
     \begin{equation*}
         2b(b^2+b+2)-(b+1)^3=(b-1)(b^2+1),
     \end{equation*}
     which is positive for $b>1$.
     Similarly,
      \begin{equation*}
         J_2(b,2)+J_3(b,2)=\log\left(\frac{27(b^2+b+4)(b^2+2b+1)}{8(b^2+4b+4)^2}\right),
     \end{equation*}
     and
     \begin{equation*}
         27(b^2+b+4)(b^2+2b+1)-8(b^2+4b+4)^2=(b-1)(19b^3+36b^2+33b+20),
     \end{equation*}
     which is positive for $b>1$.
     
     Hence $J_2(b,c)+J_3(b,c)>0$ when $b\geq 2$ and $c=1$ or $c=2$. Combining with the results for $J_1(b,c)$, and substituting in equation \eqref{eq:RL-LL}, we conclude that ${\frac{d}{db}\left( RL(b,c)-LL(b,c) \right)>0}$ for $c=1$ and $b\geq 4$, and $c=2$ and $b\geq3$ as well.

     Since the derivative is positive in all cases, it is sufficient to prove that $RL(4,1)>LL(4,1)$ and $RL(3,c)>LL(3,c).$
     The first one can be verified directly, a computation shows that $RL(4,1)-LL(4,1)\sim 0.245$ and thus $RL(4,1)>LL(4,1)$.
    The second inequality, $RL(3,c)>LL(3,c)$, is equivalent to
    $$\Delta_c(6)+\Delta_c(5)+\Delta_c(4)+2\Delta_c(1) > 4\Delta_c(3)+\Delta_c(2).$$
    This is true for every $c \ge 2$ since
    \begin{align*}
            \int_{c-1}^{c} \log\left( (t+6)(t+5)(t+4)(t+1)^2 \right) \d t >\int_{c-1}^{c}  \log\left( (t+3 )^4(t+2) \right) \d t,
        \end{align*}
        as $(t+6)(t+5)(t+4)(t+1)^2>(t+3)^4(t+2)$ for every $t \ge 1.$
\end{claimproof}

\end{document}